\documentclass[12pt, oneside, a4paper]{article}

\usepackage{amsmath}
\usepackage{amsfonts}
\usepackage{amssymb}
\usepackage{amsthm,mathrsfs}
\usepackage{enumerate}
\usepackage{graphicx}
\newtheorem{theorem}{Theorem}[section]

\theoremstyle{definition}

\newtheorem{example}[theorem]{Example}

\title{\textbf{ non-vanishing elements and complex group algebras}}
\author{Mahdi Ebrahimi\footnote{ m.ebrahimi.math@ipm.ir}
 \\
 {\small\em  School of Mathematics, Institute for Research in Fundamental Sciences (IPM)},\\{\small\em P.O. Box: 19395--5746, Tehran, Iran}}

\date{}

\begin{document}

\maketitle


\begin{abstract}
Let $G$ be a finite group, and let $\mathrm{Irr}(G)$ denote the set of irreducible complex characters of $G$. An element $x$ of $G$ is said to be vanishing, if for some $\chi$ in $\mathrm{Irr}(G)$, we have $\chi(x)=0$. Also the element $x$ is called rational if $x$ is conjugate to $x^i$ for every integer $i$ co-prime to the order of $x$. We define the weight of $G$ as $\omega(G):=(\sum_{\chi\in \mathrm{Irr}(G)}\chi(1))^2/|G|$. In this paper, we show that for every rational non-vanishing element $x\in G$, the order of  $C_G(x)$ is at least $\omega(G)$.
 \end{abstract}

\section{Introduction}
$\noindent$Let $G$ be a finite group. Denote by $\mathrm{Irr}(G)=\{\chi_1, \chi_2,\dots,\chi_k\}$ the set of all complex irreducible characters of $G$. The concept of a non-vanishing element of a finite group $G$ was introduced in \cite{in}.  An element $x\in G$ is said to be \textit{vanishing} if $\chi(x)= 0$, for some irreducible complex character $\chi$ of $G$; otherwise $x$ is called a \textit{non-vanishing} element of $G$. It is a classical theorem of Burnside \cite[Theorem 3.15]{I} that every non-linear $\chi \in \mathrm{Irr}(G)$ vanishes on some element of $G$. In other words, looking at the character table of $G$, the rows which do not contain the value $0$ are precisely those corresponding to linear characters. But it is in general not true that the columns not containing the value $0$ are precisely those corresponding to conjugacy classes of central elements, as there are finite groups having non-central non-vanishing elements. For instance every $3$-cycle of the symmetric group $\mathrm{Sym}(7)$ is a non-central non-vanishing element.

In resent decades, non-vanishing elements of finite groups have been extensively studied by researchers. For instance, Isaacs, Navarro and wolf \cite{in} showed that if $G$ is solvable and $x\in G$ is non-vanishing, then the order of $x$ modulo the Fitting subgroup $F(G)$ is a power of two, and so non-vanishing elements of odd order lie in $F(G)$. For arbitrary groups, it has been shown \cite{do} that a non-vanishing element $x$ of a finite group $G$ lies in $F(G)$ if the order of $x$ is co-prime to 6. Also Miyamoto \cite{mi} has proved that if $A$ is an elementary abelian normal $p$-subgroup of a finite group $G$ and $p$ is a Sylow $p$-subgroup of $G$, then all elements in $Z(P)\cap A$ are non-vanishing.

 It is well known that the complex group algebra $$\mathbb{C}G=\bigoplus_{i=1}^k M_{n_i}(\mathbb{C}),$$
where $n_i=\chi_i(1)$, for $1\leq i\leq k$. Thus, the complex  group algebra $\mathbb{C}G$ determines the character degrees of $G$ and their multiplicities.
An important question in character theory is whether One can recover a group or its properties from its complex group algebra. It has been shown \cite{Be} that all \textit{quasi-simple} groups are uniquely determined up to isomorphism by their complex group algebras. Recall that a finite group $G$ is quasi-simple if the group $G$ is perfect and $G/Z(G)$ is a non-abelian finite simple group. Also Gruninger \cite{Gr} proved that an element $x$ of a finite group $G$ is non-vanishing, if and only if the class sum of $x$ in the group algebra $\mathbb{C}G$ is a unit.
 For a finite group $G$ we define the \textit{weight} of $G$ as $\omega(G):=(\sum_{\chi\in \mathrm{Irr}(G)}\chi(1))^2/|G|$. It is clear that $\omega(G)$ can be evaluated by the complex group algebra $\mathbb{C}G$. In this paper, we wish to show that $\omega(G)$ is a lower bound for the order of the centralizer of each rational non-vanishing element of $G$. A rational element of the group $G$ is an element $x$ such that $x$ is conjugate to $x^i$ for every integer $i$ co-prime to the order of $x$.
\begin{theorem} \label{main}
Let $G$  be a finite group and let $x\in G$ be a rational non-vanishing element. Then  $|C_G(x)|\geq\omega(G)$.
\end{theorem}

\section{The energy of Cayley graphs}
$\noindent$In this paper, all groups and graphs are assumed to be finite.
 Let $\Gamma$ be a simple graph with vertex set $\{\nu_1,\nu_2,\dots, \nu_n\}$.
 The \textit{adjacency matrix} of $\Gamma$, denoted by $A(\Gamma)$, is the $n\times n$ matrix such that the $(i,j)$-entry is $1$ if $\nu_i$ and $\nu_j$ are adjacent, and is $0$ otherwise.
  The \textit{eigenvalues} of $\Gamma$ are the eigenvalues of its adjacency matrix $A(\Gamma)$.

   Let $G$ be a finite group and $S$ be
 an inverse closed subset of $G$ with $1 \notin S$.
  The \textit{Cayley graph} $\mathrm{Cay}(G,S)$ is the graph which has the elements of
   $G$ as its vertices and two vertices $u,\nu \in G$
    are joined by an edge if and only if $\nu=au$, for some $a\in S$. A Cayley graph $\mathrm{Cay}(G,S)$ is called \textit{normal} if $S$ is closed under conjugation with elements of $G$. It is well known that the eigenvalues of a normal Cayley graph  $\mathrm{Cay}(G,S)$ can be expressed in terms of the irreducible characters of $G$ \cite[p.235]{eigen}.

\begin{theorem}\label{eigen}(\cite{2}, \cite{6}, \cite{17}, \cite{19})
The eigenvalues of a normal Cayley graph $\mathrm{Cay}(G,S)$
are given by $\eta_\chi=\frac{1}{\chi(1)}\sum_{a\in S}\chi(a)$ where
 $\chi$ ranges over all complex irreducible characters of $G$. Moreover,
  the multiplicity of $\eta_{\chi}$ is $\chi(1)^2$.
\end{theorem}

Suppose $\{\lambda_1, \lambda_2, \dots, \lambda_n\}$ is the set of all eigenvalues of a simple graph $\Gamma$.
 The \textit{energy} of $\Gamma$ is defined as $\mathcal{E}(\Gamma):=\sum_{i=1}^n|\lambda_i|$.
  This concept was first introduced by Gutman \cite{182}. Now we wish to obtain a lower bound for the energy of normal Cayley graphs.
  \begin{theorem}\label{energy}
  Assume that $G$ is a finite group and $C$ is a conjugacy class of $G$ containing a non-trivial rational non-vanishing element $x\in G$. Then $$\mathcal{E}(\mathrm{Cay}(G,C))\geq |C|\sum_{\chi \in \mathrm{Irr}(G)}\chi(1).$$
  \end{theorem}
  \begin{proof}
 Let $\chi \in \mathrm{Irr}(G)$. Then by \cite[Corollary 3.6]{I}, $\chi(x)$ is an algebraic integer.  Also as $x$ is a rational element of $G$, using \cite[Theorem 1]{ko}, we deduce that eigenvalues of $\mathrm{Cay}(G,C)$ are integers. Thus as $x$ is non-vanishing, applying Theorem \ref{eigen}, we conclude that $\chi(x)$ is a non-zero integer. Hence by Theorem \ref{eigen}, we get
 \begin{align}
\mathcal{E}(\mathrm{Cay}(G,C))&=\sum_{\chi\in \mathrm{Irr}(G)}\chi(1)|C||\chi(x)|\nonumber\\
&\geq |C|\sum_{\chi\in \mathrm{Irr}(G)}\chi(1).\nonumber
\end{align}
  \end{proof}
  \section{The proof of Theorem \ref{main}}
  In this section, we wish to prove Theorem \ref{main}.\\
  \textit{Proof of Theorem \ref{main}:} Without loss of generality, we can assum that $x$ is non-trivial. Let $C$ be the conjugacy class of $G$ containing $x$. Set $\Gamma:=\mathrm{Cay}(G,C)$. By Theorem \ref{energy},
  \begin{align}\label{(1)}
  \mathcal{E}(\Gamma)\geq |C|\sum_{\chi \in \mathrm{Irr}(G)}\chi(1).
  \end{align}
  Since $x$ is non-vanishing, the adjacency matrix $\mathrm{A}(\Gamma)$ of $\Gamma$ is non-singular. Thus using \cite[Theorem 4.5]{Li} and inequality (\ref{(1)}), we deduce that
  $$|G|\sqrt{|C|}\geq |C|\sum_{\chi \in \mathrm{Irr}(G)}\chi(1).$$
  Therefore
  $$|C_G(x)|\geq \omega(G).$$
 This completes the proof.\qed

  We end this section with the following example.
  \begin{example}
   Let $G$ be the sporadic simple group $M_{11}$ and let $x$ be a $5$-element of $G$ (see \cite{atlas}). Since
   $$\mathbb{C}G=M_1(\mathbb{C})\oplus M_{10}(\mathbb{C})^3\oplus M_{11}(\mathbb{C})\oplus M_{16}(\mathbb{C})^2\oplus  M_{44}(\mathbb{C})\oplus M_{45}(\mathbb{C})\oplus M_{55}(\mathbb{C}),$$
   it is easy to see that $\omega(G)=6.00\bar{05}$. Thus as $x$ is a rational element with $|C_G(x)|=5$, by Theorem \ref{main}, $x$ is a vanishing  element of $G$.
  \end{example}

\section*{Acknowledgements}
This research was supported in part
by a grant  from School of Mathematics, Institute for Research in Fundamental Sciences (IPM).



\begin{thebibliography}{22}
\bibitem{2}
L. Babai, Spectra of Cayley graphs, J. Comb. Theory, Ser. B 27(1979) 180-189. DOI: https://doi.org/1001016/0095-8956(79)-90079-0
%
\bibitem{Be}
C. Bessenrodt, H. N. Nguyen, J. B. Olsson, H. P. Tong-Viet, Complex group algebras of the double covers of the symmetric and alternating groups, Algebra Number Theory 9(3) (2015)601-628.
%
\bibitem{atlas}
J.H. Conway, R.T. Curtis, S.P. Norton, R.A. Parker, R.A. Wilson, Atlas of finite groups, Clarendon Prees, Oxford 1985.
%
\bibitem{eigen}
C. W. Curtis, I. Reiner, Representation theory of finite groups and associative algebras, Pure and Applied Mathematics, vol. XI, Wiley, New york, 1962.
%
\bibitem{6}
P. Diaconis, M. Shahshahani, Generating a random permutation with random transpositions, Z. Wahrsch. Verw. Gebiete. 57(1981) 159-179. 0044-3719/81/0057/0159/S04.20
%
\bibitem{do}
S. Dolfi, G. Navarro, E. Pacifici, L. Sanus, P. H. Tiep, Non-vanishing elements of finite groups, J. Algebra 323 (2010), no.2, 540-545.
%
\bibitem{Gr}
M. Gruninger, Two remarks about non-vanishing elements in finite groups, J. Algebra, 460 (2016) 366-369.
%
\bibitem{182}
I. Gutman, The energy of a graph, Ber. Math. Statist. Sekt. Forschungszentrum Graz 103 (1978) 1-22.
%
\bibitem{I}
I.M. Isaacs, character Theory of Finite Groups, Dover Publications, New York, 1994.
%
\bibitem{in}
I.M. Isaacs, G. Navarro, T.R. Wolf, Finite group elements where no irreducible characeter vanishes, J. Algebra 222 (1999) 413-423.
%
\bibitem{ko}
E. V. Konstantinova, D. Lytkina, Integral cayley graphs over finite groups, Algebra Colloq. 27(1) (2020) 131-136.
%
\bibitem{Li}
 X. Li, Y. Shi, I. Gutman, Graph Energy,	Springer, New York, 2012.
	%
\bibitem{17}
A. Lubotzky, Discrete groups, Expanding graphs, and invariant Measures, Birkhauser Verlag, Basel, 1994.
%
\bibitem{mi}
M. Miyamoto, Non-vanishing elements in finite groups, J. Algebra, 364 (2012) 88-89.
%
\bibitem{19}
M. Ram Murty, Ramanujan graphs, J. Ramanujan Math. soc. 18(2003)1-20.
%
\end{thebibliography}
\end{document}